\documentclass[10pt]{article}
\usepackage[english]{babel}
\usepackage{amsmath}
\usepackage{amsthm}
\usepackage{amssymb}

\newtheorem{theorem}{Theorem}[section]
\newtheorem{proposition}[theorem]{Proposition}

\newtheorem{lemma}[theorem]{Lemma}
\newtheorem{remark}[theorem]{Remark}

\newcommand{\mL}{{\cal L}}

\newcommand{\R}{\mathbb{R}}

\newcommand{\beq}{\begin{equation}}
\newcommand{\eeq}{\end{equation}}

\renewcommand{\k}{{\cal C}}
\newcommand{\C}{\mathbb{C}}
\newcommand{\N}{\mathbb{N}}
\newcommand{\Z}{{\mathbb{Z}}}

\newcommand{\mM}{{\cal M}}

\title {\bf Classical solutions to quasilinear parabolic problems with dynamic boundary conditions
\thanks{This work was partially supported by the 
MIUR-PRIN Grant 20089PWTPS
{\it Analisi Matematica nei Problemi Inversi per le Applicazioni}. The author is also member of the GNAMPA of Istituto Nazionale di Alta Matematica. } }

\author{Davide Guidetti\\
Dipartimento di Matematica, Universit\`{a} di Bologna\\
Piazza di Porta San Donato 5, 40126 Bologna, Italy
\\
davide.guidetti@unibo.it}
\date{}

\begin{document}

\maketitle

\abstract{We study linear nonautonomous parabolic systems with dynamic boundary conditions. Next, we apply these results to show a theorem of local existence and uniqueness of a classical solution to a second order quasilinear system with nonlinear dynamic boundary conditions. }

\medskip


\medskip


\section{Introduction and preliminaries}\label{se1}

The main aim of this paper is to study existence and uniqueness of local classical solutions to quasilinear parabolic systems with dynamic boundary conditions in the form
\begin{equation}\label{eq5.1}
\left\{\begin{array}{ll}
D_t u(t,x) = \sum_{|\alpha| = 2}Êa_\alpha(t,x, u(t,x), \nabla_x u(t,x)) D_x^\alpha u(t,x) + f(t,x, u(t,x), \nabla_x u(t,x)), & t \geq 0, x \in \Omega, \\ \\
D_t u(t,x') + \sum_{j=1}^nÊb_j(t,x', u(t,x')) D_{x_j} u(t,x') = h(t,x, u(t,x)), & t \geq 0,  x' \in \partial \Omega, \\ \\
u(0,x) = u_0(x), & x \in \Omega. 
\end{array}
\right.
\end{equation}
With the term "classical solutions", we mean solutions possessing all the derivatives appearing in (\ref{eq5.1}) in pointwise sense: so we look for conditions guaranteeing the existence of solutions $u(t,x)$ with 
$D_tu$, $D_x^\alpha u$ ($|\alpha| \leq 2$) continuous in $[0, \tau] \times \overline \Omega$, for some $\tau > 0$. If the functions $b_j$ and $h$ are suitably regular, these conditions, together with the second equation in (\ref{eq5.1}) (the boundary condition) imply that the map $t \to D_t u(t,\cdot)$ should be continuous with values in $C^1(\partial \Omega)$. However, it is well known that, in order to get neat results for parabolic problems, it is often advisable to replace continuous  with  H\"older continuous functions. So a natural class of solutions could be the set of functions in $C^{1+\beta/2, 2+\beta}((0, \tau) \times \Omega)$ (see (\ref{eq1.6})), such that the restriction of $D_tu$ to $[0, \tau] \times \partial \Omega$ is bounded (as a function of $t$) with values in $C^{1+\beta}(\partial \Omega)$. 

In order to frame our results, I begin by recalling some previous literature, concerning nonlinear parabolic problems with dynamic boundary conditions. The first paper I quote is \cite{Ka1}: here the author considers the system 
\begin{equation}\label{eq1.1}
\left\{\begin{array}{lll}
D_t u - \sum_{j=1}^n D_{x_j}[a_{j}(u,\cdot, \nabla_x u) D_{x_j} u]Ê+ a_0(u,\cdot) = f(t,\cdot)  & {\rm in  }  & \R^+ \times \Omega,Ê\\ \\
D_t u + \sum_{j,k=1}^n a_{jk}(u, \cdot) \nu_j D_{x_k} u + b_0(u,\cdot)u = g_1(u) & {\rm on} & \R^+ \times \Gamma_1, \\ \\
\sum_{j,k=1}^n a_{jk}(u, \cdot) \nu_j D_{x_k} u + b_1(u,\cdot) = 0 & {\rm on} & \R^+ \times \Gamma_2, \\ \\
u(0,\cdot) = u_0 & {\rm in}  & \Omega,
\end{array}
\right. 
\end{equation}
where I have indicated with $\nu(x')$ the unit normal vector to $\partial \Omega$ in $x' \in \partial \Omega$, pointing outside $\Omega$. Employing monotonicity assumptions and Rothe's method,  imposing a polynomial growth of the coefficients, he constructs a generalized solution. The initial datum $u_0$ is taken in a suitable space $W^{1,p}(\Omega)$, with $p$ connected with the growth conditions of the functions $a_j$.  Some results of regularity are proved (for example, the solution in continuos with values in $H^2(\Omega')$ for every $\Omega'$ with compact closure in $\Omega$). 

In the paper \cite{Es1}, the author considers the system 
\begin{equation}\label{eq1.1}
\left\{\begin{array}{lll}
D_t u - \sum_{j,k=1}^n D_{x_j}[a_{jk}(u,\cdot) D_{x_k} u]Ê+ \sum_{j=1}^n a_j(u,\cdot) D_{x_j} u + a_0(u,\cdot)u = f(u)  & {\rm in  }  & \R^+ \times \Omega,Ê\\ \\
\epsilon D_t u + \sum_{j,k=1}^n a_{jk}(u, \cdot) \nu_j D_{x_k} u + b_0(u,\cdot) = g_1(u) & {\rm on} & \R^+ \times \Gamma_1, \\ \\
\sum_{j,k=1}^n a_{jk}(u, \cdot) \nu_j D_{x_k} u + b_0(u,\cdot)u = g_1(u) & {\rm on} & \R^+ \times \Gamma_2, \\ \\
u(0,\cdot) = u_0 & {\rm in}  & \Omega. 
\end{array}
\right. 
\end{equation}
Here $u: [0, \infty) \times \Omega \to \R^N$ ($a_{jk}(u,\cdot)$,  $a_0(u,\cdot)$, $b_0(u,\cdot)$ are, in fact, matrixes),  $\sum_{j,k=1}^n D_{x_j}[a_{jk}(\xi,\cdot) D_{x_k}]$ are normally elliptic and all the functions appearing in the first equation in (\ref{eq1.1}) are smooth ($C^\infty$). If $p > n$, 
$n/p < r < 1 < \tau < 1+ 1/p$, $\delta = \frac{1-r}{2}$, $\sigma \in \R^+$, $u_0 \in W^{\tau,p}(\Omega)$, there is a unique maximal weak solution $u$ in $C(J; W^{1,p}(\Omega)) \cap C^\delta(J; W^{r,p}(\Omega)) \cap C(J \setminus \{0\}; W^{\sigma,p}(\Omega))$, with $J = [0, T]$, for some $T > 0$. No growth conditions are imposed to the coefficients.  Related results are sketched in \cite{Hi1}. 

Replacing in the boundary condition the time derivative with the second term of the parabolic equation, one is formally reduced to a stationary boundary condition of second order, which is usually called "generalized Wentsell boundary condition". So, in \cite{FaGoGoRo2}Ê the authors consider  the operator $\mathcal A u = \phi(x,u'(x)) u''(x) + \psi(x,u(x),u'(x))$, with suitable assumptions on $\phi$ and $\psi$, in the interval $[0, 1]$. General boundary conditions in the form $B(u)(j):= \alpha_j (Au)(j) + \beta_j u'(j) \in \gamma_j(u(j))$ $(j \in \{0, 1\})$, with $\gamma_j$ maximal monotone are imposed. Then it is proved that $\mathcal A$, equipped  with such conditions, is 
$m-$dissipative and generates a nonlinear contraction semigroup in $C([0, 1])$. 

In \cite{Wa2}, the author considers a domain $\Omega$  with Lipschitz boundary, and the Laplacian with the Wentzell-Robin boundary condition
$$
\Delta u + \frac{\partial u}{\partial \nu} + \beta u = 0.
$$
Here $\beta \in L^\infty(\partial \Omega)$, $\beta \geq 0$. Then he proves that this operator generates an analytic semigroup in $C(\overline \Omega)$.

The same author   considers in \cite{Wa1} the $p-$Laplacian $A_pu = div(a(x) |\nabla u|^{p-2}) \nabla u)$, with $a(x) > 0$ in $\Omega$ and the nonlinear Wentzell boundary condition on $\partial \Omega$ 
$$0 \in A_pu + b |\nabla u|^{p-2} \frac{\partial u}{\partial \nu}Ê+ \beta(\cdot,u), $$
Here $\beta(x,\cdot)$ is the sub differential of the functional $B(x,\cdot)$. He proves that it generates a nonlinear sub-Markovian $C_0-$semigroup on suitable $L^2-$spaces and, in case $b(x) \geq b_0 > 0$,  he obtains the existence of nonlinear, non expansive semigroups in $L^q$ spaces, for every $q \in [1, \infty)$. A related situation with dynamic boundary conditions is treated in \cite{GaWa1}. 
 
Finally, the asymptotic behavior of semilinear parabolic system with dynamic boundary conditions is studied in \cite{CoGoGo1}. The nonlinearity is only in the boundary condition. A generalization (with a first order nonlinear term in the parabolic equation) is given in \cite{CoGoGo2}. 

In this paper, differently from these papers, we want to show the existence and uniqueness of local solutions to (\ref{eq5.1}) which are regular up to $t = 0$. Observe that we consider the case that the coefficients of the elliptic operator in the parabolic equation depend also on $\nabla u$. Moreover, we consider systems which are not necessarily in divergence form and no particular connection between the elliptic operator and the first order operator $\sum_{j=1}^nÊb_j(t,x', u) D_{x_j}$ is required: the only structural condition that we impose is that $\sum_{j=1}^nÊb_j(t,x', u) \nu_j(x') > 0$ for every $x' \in \partial \Omega$. Finally, we do not impose any kind of growth condition on the coefficients. 

The organization of this paper is the following: in this first section, we introduce some notations,  together with some known facts,  which we are going to employ in the sequel. The second section contains a careful study of linear
non autonomous parabolic systems, which may have some interest in itself, in particular Theorem \ref{th4.5}. This study is preliminary to the final third section, containing the main result, Theorem \ref{th3.5}. Such theorem states the existence and uniqueness of a local solution $u$ to (\ref{eq5.1}) such that
$u \in C^{1+\beta/2, 2+\beta}((0, \tau) \times \Omega)$, with $D_tu_{|(0, \tau) \times \partial \Omega}$ bounded with values in $C^{1+\beta}(\partial \Omega)$ ($\beta \in (0, 1)$). Apart some regularity of the coefficients, we require only the strong ellipticity
of the operator $\sum_{|\alpha| = 2} a_\alpha(t,x,u,p) D_x^\alpha$ and the condition $\sum_{j=1}^n b_j(t,x', u) \nu_j(x') > 0$ if $x' \in \partial \Omega$. Concerning the initial datum $u_0$, it should belong to $C^{2+\beta}(\Omega)$ and a certain (necessary) compatibility condition ((\ref{eq3.2})) should hold. 

The main tool of the proof is a certain maximal regularity result which was
proved in \cite{Gu4}, and it is stated in Theorem \ref{th2.1}. 

We pass to the aforementioned notations and known facts.

$C(\alpha, \beta, \dots)$ will indicate a positive real number depending on $\alpha, \beta, \dots$ and may be different from time to time. The symbol $\nabla_{x,u}b$ will indicate the gradient of $b$ with respect to the vector $(x,u)$ ($x \in \R^n, u \in \R)$. On the other hand, $D^2_{x_j u}b$ stands for $\frac{\partial^2 b}{\partial x_j \partial u}$. 

Let $\Omega$ be an open subset of $\R^n$. We shall indicate with 
$\mathcal C(\Omega)$ the class of complex valued continuous functions and with $C(\Omega)$ the subspace of uniformly continuous and bounded functions. 
If $f \in C(\Omega)$, it is continuously extensible to its topological closure $\overline \Omega$. We shall identify $f$ with this extension. If $m \in \N$, we indicate with $\mathcal C^m(\Omega)$($C^m(\Omega)$) the class of functions $f$ in $\mathcal C(\Omega)$($C(\Omega)$), whose derivatives $D^\alpha f$, with order $|\alpha| \leq m$, belong to $\mathcal C(\Omega)$($C(\Omega)$). $C^m(\Omega)$ admits the natural norm
\begin{equation}
\|f\|_{C^m(\Omega)}Ê:= \max\{\|D^\alpha f\|_{C(\Omega)}Ê: |\alpha| \leq m\},
\end{equation}
with $\|f\|_{C(\Omega)} := \sup_{x \in \Omega}Ê|f(x)|$. If 
$$[f]_{C^\beta(\Omega)}: = \sup_{x,y \in \Omega, x \neq y}Ê|x-y|^{-\beta} |f(x) - f(y)|\}$$ 
and $m \in \N_0$, we set 
\begin{equation}
\|f\|_{C^{m+\beta}(\Omega)}:= \max\{\| f\|_{C^m(\Omega)}, \max\{ [D^\alpha f]_{C^\beta(\Omega)}Ê: |\alpha| = m\}\},
\end{equation}
and, of course, $C^{m+\beta}(\Omega) = \{f \in C^m(\Omega): \|f\|_{C^{m+\beta}(\Omega)} < \infty\}$. If $\partial \Omega$ is sufficiently regular, $0 \leq \beta_0 < \beta < \beta_1$,
$$C^{\beta}(\Omega) \in J_{\frac{\beta - \beta_0}{\beta_1 - \beta_0}}(C^{\beta_0}(\Omega), C^{\beta_1}(\Omega)),$$
which means that there exists $C > 0$, such that, $\forall f \in C^{\beta_1}(\Omega)$
\begin{equation}\label{eq1.5}
\|f\|_{C^{\beta}(\Omega)}  \leq C \|f\|_{C^{\beta_0}(\Omega)}^{1-\theta}  \|f\|_{C^{\beta_1}(\Omega)}^{\theta}. 
\end{equation}

We pass to consider vector valued functions. Let $X$ be a Banach space. If $A$ is a set, we shall indicate with
$B(A; X)$ the Banach space of bounded functions from $A$ to $X$. If $m \in \N_0$, $\beta \geq 0$ and $\Omega$ is an open subset of $\R^n$, the definitions of $\mathcal C^m(\Omega; X)$, $C^m(\Omega; X)$, $C^\beta(\Omega; X)$ and of the norms $\|\cdot\|_{C^\beta(\Omega; X)}$ can be  obtained by obvious modifications of the corresponding, in the case $X = \C$. (\ref{eq1.5}) can be generalized to vector valued functions. 

 If $I$ is an open interval in $\R$, $\Omega$ is an open subset of $\R^n$ and $\alpha, \beta$ are nonnegative, we set
\begin{equation}\label{eq1.6}
C^{\alpha,\beta}(I \times \Omega):= C^\alpha(I; C(\Omega)) \cap C^\beta(\Omega; C(I)),
\end{equation}
equipped with its natural norm
$$
\|f\|_{C^{\alpha,\beta}(I \times \Omega)}:= \max\{\|f\|_{C^\alpha(I; C(\Omega))}, \|f\|_{B(I; C^\beta(\Omega))}\}. 
$$
The following inclusions hold (see \cite{Gu4}): 

\begin{lemma}\label{le2.7B} (I) $C^\beta(\Omega; C(I)) \subseteq  B(I; C^\beta(\Omega))$. 

(II) Suppose $\alpha, \beta \geq 0$ with $\beta \not \in \Z$ and $\Omega$ such that there exists a common linear bounded extension operator, mapping $C(\Omega)$ into $C(\R^n)$ and
$C^\beta(\Omega)$ into $C^\beta(\R^n)$. Then, 
$$C^{\alpha,\beta}(I \times \Omega) = C^{\alpha}(I; C(\Omega)) \cap B(I; C^\beta(\Omega)). $$
Let $\beta \in (0,1)$ and suppose that   there exists a common linear bounded extension operator mapping $C^\gamma(\Omega)$ into $C^\gamma(\R^n)$, $\forall \gamma \in [0, 2+\beta]$. 
Then

(III) $C^{1+\beta/2, 2+\beta}(I \times \Omega) = C^{1+\beta/2}(I; C(\Omega)) \cap B(I; C^{2+\beta}(\Omega))$; 

(IV) if $f \in C^{1+\beta/2, 2+\beta}(I \times \Omega)$, $D_t f \in B(I; C^\beta(\Omega))$; 

(V) $C^{1+\beta/2, 2+\beta}(I \times \Omega) \subseteq C^{\frac{1+\beta}{2}}(I; C^1(\Omega)) \cap C^{\beta/2}(I; C^2(\Omega))$.

\end{lemma}
We shall need also spaces $C^{\alpha,\beta}(I \times V)$, with $V$ suitably regular submanifold of $\R^n$: we shall consider, in particular, the case $V = \partial \Omega$, with $\Omega$ open, bounded subset of $\R^n$. Of course, in this case $C^\beta(V; C(I))$ can be defined by local charts. 

We shall employ the following version of the continuation method: 

\begin{proposition}\label{pr1.20B}
Let $X, Y$ be Banach spaces and $L \in C([0, 1]; \mL(X,Y))$. Assume the following:

(a) there exists $M \in \R^+$, such that, $\forall x \in X$, $\forall \epsilon \in [0, 1]$,
$$
\|x\|_X \leq M \|L(\epsilon) x\|_Y; 
$$
(b) $L(0)$ is onto $Y$.

Then, $\forall \epsilon \in [0, 1]$ $L(\epsilon)$ is a linear and topological isomorphism between $X$ and $Y$. 
\end{proposition}

\section{Nonautonomous linear systems}

\setcounter{equation}{0}

In this section we shall consider the following system

\begin{equation}\label{eq4.1}
\left\{\begin{array}{l}
D_t u(t,x) - A(t,x,D_x) u(t,x) = f(t,x), \quad t \in (0, T), x \in \Omega, \\ \\
D_t u(t,x') + B(t,x',D_x) u(t,x') = h(t,x'), \quad t \in (0, T),  x' \in \partial \Omega, \\ \\
u(0, x) = u_0(x), \quad x \in \Omega,
\end{array}
\right. 
\end{equation}
with the following conditions:

\medskip
{\it (AF1) $\Omega$ is an open bounded subset of $\R^n$, lying on one side of its boundary $\partial \Omega$, which is a submanifold of class $C^{2+\beta}$ of $\R^n$, for some $\beta \in (0, 1)$; 

(AF2) $A(t,x,D_x) = \sum_{|\alpha| \leq 2} a_\alpha(t,x) D_x^\alpha$, $\sum_{|\alpha| \leq 2} \|a_\alpha\|_{C^{\beta/2,\beta}((0, T) \times \Omega)}Ê\leq N$; if $|\alpha| = 2$, $a_\alpha$ is real valued and $\sum_{|\alpha| = 2} a_\alpha(t,x) \eta^\alpha \geq \nu |\eta|^2$, for some $N, \nu \in \R^+$, $\forall (t,x) \in [0, T] \times \overline \Omega$, $\forall \eta \in \R^n$; 

(AF3) $B(t,x',D_x)  = \sum_{|\alpha| \leq 1} b_\alpha(t,x) D_x^\alpha$, $\sum_{|\alpha| \leq 1} \|b_\alpha\|_{C^{\beta/2,1+\beta}((0, T) \times \partial \Omega)} \leq N$; if $|\alpha| = 1$, $b_\alpha$ is real valued and $\sum_{|\alpha| = 1} b_\alpha(t,x') \nu(x')^\alpha \geq \nu$ $\forall x' \in \partial \Omega$. 
}

\medskip

In order to study (\ref{eq4.1}), we consider the autonomous system
\begin{equation}\label{eq2.2}
\left\{\begin{array}{l}
D_t u(t,x) - A(x,D_x) u(t,x) = f(t,x), \quad t \in (0, T), x \in \Omega, \\ \\
D_t u(t,x') + B(x',D_x) u(t,x') = h(t,x'), \quad t \in (0, T),  x' \in \partial \Omega, \\ \\
u(0, x) = u_0(x), \quad x \in \Omega,
\end{array}
\right. 
\end{equation}
We recall the following result, which was proved in \cite{Gu4}: 

\begin{theorem}\label{th2.1}
Consider  system (\ref{eq2.2}), with the following conditions (AE1)-(AE3):

(AE1) (AF1) holds; 

(AE2)  $a_\alpha \in C^{\beta}(\Omega)$, $\forall \alpha \in \N_0^n$, $|\alpha| \leq 2$;  if $|\alpha| = 2$, $a_\alpha$ is real valued and $\sum_{|\alpha| = 2} a_\alpha(x) \eta^\alpha \geq \nu |\eta|^2$, for some $\nu \in \R^+$, $\forall x \in  \overline \Omega$, $\forall \eta \in \R^n$; 

(AE3) $B(x',D_x)  = \sum_{|\alpha| \leq 1} b_\alpha(x) D_x^\alpha$, $b_\alpha \in C^{1+\beta} (\partial \Omega)$ $\forall \alpha \in \N_0^n$, $|\alpha| \leq 1$; if $|\alpha| = 1$, $b_\alpha$ is real valued and 
$\sum_{|\alpha| = 1} b_\alpha(x') \nu(x')^\alpha > 0$ $\forall x' \in \partial \Omega$. 

Then the following conditions are necessary and sufficient, in order that (\ref{eq2.2}) have a unique solution $u$ in $C^{1+\beta/2,2+\beta}((0, T) \times \Omega)$, with $D_tu_{|(0, T) \times \partial \Omega)}$ in $B((0, T); C^{1+\beta}(\partial \Omega))$:

(a) $f \in C^{\beta/2,\beta}((0, T) \times \Omega))$; 

(b) $h \in C^{\beta/2,1+\beta}((0, T) \times \partial \Omega)$; 

(c) $u_0 \in C^{2+\beta}(\Omega)$; 

(d) $\forall \xi' \in \partial \Omega$
$$
A(\xi',D_\xi) u_0(\xi') +  f(0,\xi') = -B(\xi',D_\xi) u_0(\xi') + h(0,\xi').
$$
\end{theorem}

\begin{lemma}\label{le4.1}
Assume that (AE1)-(AE3), $T_0 \in \R^+$ and $T \leq T_0$. Suppose that $f$, $h$, $u_0$ satisfy conditions (a)-(d) in the statement of Corollary \ref{co2.2}. Let $u$ be the solution in $C^{1+\beta/2,2+\beta}((0, T) \times \Omega)$, with $D_t u \in B((0, T); C^{1+\beta}(\partial \Omega))$
 of (\ref{eq2.2}). Then: 
 
 (I) there exists $C(T_0,A,B)$ in $\R^+$, such that
$$
\begin{array}{c}
\|u\|_{C^{1+\beta/2,2+\beta}((0, T) \times \Omega))} + \|u\|_{C^{\beta/2}((0, T); C^2(\Omega))} + \|u\|_{C^{\frac{1+\beta}{2}}((0, T); C^1(\Omega))}Ê\\ \\
\leq C(T_0,A,B) (\|f\|_{C^{\beta/2,\beta}((0, T) \times \Omega)} + \|u_0\|_{C^{2+\beta}(\Omega)}Ê+ \|h\|_{C^{\beta/2,1+\beta}((0, T) \times \partial \Omega)}). 
\end{array}
$$
(II) Suppose that $u_0 = 0$. Then, if $0 \leq \theta \leq 1$, 
\begin{equation}\label{eq4.2}
\|u\|_{C^\theta((0, T); C(\Omega))}Ê\leq C(T_0,A,B) T^{1-\theta}Ê(\|f\|_{C^{\beta/2,\beta}((0, T) \times \Omega)} + \|h\|_{C^{\beta/2,1+\beta}((0, T) \times \partial{\Omega})}),
\end{equation}
if $0 \leq \theta \leq 2+\beta$, 
\begin{equation}\label{eq4.3}
\|u\|_{B((0, T); C^\theta(\Omega))}Ê\leq C(T_0,A,B,\theta) T^{\frac{2+\beta-\theta}{2+\beta}} (\|f\|_{C^{\beta/2,\beta}((0, T) \times \Omega)} + \|h\|_{C^{\beta/2,1+\beta}((0, T) \times \partial \Omega)}). 
\end{equation}
(III) Again supposing $u_0 = 0$, we have 
$$
\|u\|_{C^{\beta/2}((0, T); C^1(\Omega))}Ê\leq C(T_0,A,B) T^{1/2}Ê(\|f\|_{C^{\beta/2,\beta}((0, T) \times \Omega)} + \|h\|_{C^{\beta/2,1+\beta}((0, T) \times \partial \Omega)}).
$$
\end{lemma}

{\it Proof } We extend $f$ and $h$ to elements $\tilde f$ and $\tilde h$ in (respectively) $C^{\beta/2, \beta}((0, T_0) \times \Omega)$ and $C^{\beta/2,1+\beta}$ $((0, T_0) \times \partial \Omega)$: we set
$$
\tilde f(t,x) = \left\{\begin{array}{ll}
f(t,\xi) & \mbox{ if }Ê0 \leq t \leq T, \xi \in \Omega, \\ \\
f(T,\xi) & \mbox{ if }Ê0 \leq  T \leq t \leq T_0, \xi \in \Omega,
\end{array}
\right.
$$
$$
\tilde h(t,\xi') = \left\{\begin{array}{ll}
h(t,\xi') & \mbox{ if }Ê0 \leq t \leq T, \xi' \in \partial \Omega, \\ \\
h(T,\xi') & \mbox{ if }Ê0 \leq  T \leq t \leq T_0, \xi' \in \partial \Omega. 
\end{array}
\right.
$$
We denote with $\tilde u$ the solution to 
$$
\left\{\begin{array}{l}
D_t \tilde u(t,\xi) = A(\xi,D_\xi) \tilde u(t,\xi) + \tilde f(t,\xi), \quad t \in [0, T_0], \xi \in \Omega, \\ \\
\tilde u(0, \xi) = u_0(\xi), \quad \xi \in \Omega, \\ \\
D_t \tilde u(t,\xi') + B(\xi',D_{\xi})\tilde u(t,\xi') =  \tilde h(t,\xi'), \quad t \in [0, T_0], \xi' \in \partial \Omega. 
\end{array}
\right.
$$
Clearly, $\tilde u$ is an extension of $u$. So
$$
\begin{array}{c}
\|u\|_{C^{1+\beta/2, 2+\beta}((0, T) \times \Omega)}Ê+ \|D_t u_{|(0, T) \times \partial \Omega}\|_{B((0, T); C^{1+\beta}(\partial \Omega)} + \|u\|_{C^{\beta/2}((0, T); C^2 (\Omega))} \\ \\
\leq \|\tilde u\|_{C^{1+\beta/2, 2+\beta}((0, T_0) \times \Omega)}Ê+ \|D_t \tilde u_{|(0, T) \times \partial \Omega}\|_{B((0, T_0); C^{1+\beta}(\partial \Omega)} + \|\tilde u\|_{C^{\beta/2}((0, T_0); C^2 (\Omega))}  \\ \\
Ê\leq C(T_0,A,B)(\|\tilde f\|_{C^{\beta/2, \beta}((0, T_0) \times \Omega)} + \|u_0\|_{C^{2+\beta}(\Omega)} + \|\tilde h\|_{C^{\beta/2,1+\beta}((0, T_0) \times \partial \Omega)}) \\ \\
= C(T_0,A,B)(\|f\|_{C^{\beta/2, \beta}((0, T) \times \Omega)} + \|u_0\|_{C^{2+\beta}(\Omega)} + \|h\|_{C^{\beta/2,1+\beta}((0, T) \times \partial \Omega)}). 
\end{array}
$$
So (I) is proved.

Concerning (II), we begin by considering the case $\theta = 0$. If $0 \leq t \leq T$, we have
$$
u(t,\cdot) = \int_0^t D_s u(s,\cdot) ds, 
$$
so that
$$
\begin{array}{c}
\|u(t,\cdot)\|_{C(\Omega)} \leq \int_0^t \|D_s u(s,\cdot)\|_{C(\Omega)} ds \leq t \|u\|_{C^{1+\beta/2}((0, T); C(\Omega))} \\ \\
 \leq C(T_0,A,B) T (\|f\|_{C^{\beta/2, \beta}((0, T) \times \Omega)} + \|h\|_{C^{\beta/2,1+\beta}((0, T) \times \partial \Omega)}),
\end{array}
$$
employing (I). The cases $0 < \theta \leq 1$ follow from the foregoing from $\theta = 0$ and the fact that $C^\theta((0, T); C(\Omega)) \in J_\theta(C((0, T); C(\Omega)),C^1((0, T); C(\Omega))$. 
So (\ref{eq4.2}) is proved. Finally (\ref{eq4.3}) follows from (\ref{eq4.2}) with $\theta = 0$, (I) and the fact that $C^\theta(\Omega)) \in J_{\theta/(2+\beta)}Ê(C(\Omega), C^{2+\beta}(\Omega))$. 

It remains to consider (III). We have
$$
\begin{array}{c}
\|u\|_{C^{\beta/2}((0, T); C^1(\Omega))} = \|u\|_{C((0, T); C^1(\Omega))} + [u]_{C^{\beta/2}((0, T); C^1(\Omega))}  \\ \\
\leq \|u\|_{C((0, T); C^1(\Omega))}  + T^{1/2} [u]_{C^{(1+\beta)/2}((0, T); C^1(\Omega))},
\end{array}
$$
and the conclusion follows from (I) and (II). 

$\Box$

\begin{lemma}\label{le4.2}
We consider a system in the form
\begin{equation}\label{eq4.4}
\left\{\begin{array}{l}
D_t u(t,\xi) - A(\xi,D_\xi) u(t,\xi) - \mathcal A(t) u(t,\cdot)(\xi) = f(t,\xi), \quad t \in (0, T), \xi \in \Omega, \\ \\
D_t u(t,\xi') + B(\xi',D_\xi) u(t,\xi') + \mathcal B(t)u(t,\cdot)(\xi') = h(t,\xi'), \quad t \in (0, T),  \xi' \in \partial \Omega, \\ \\
u(0, \xi) = u_0(\xi), \quad \xi \in \Omega,
\end{array}
\right. 
\end{equation}
with the following conditions:

(a) (AE1)-(AE3) are satisfied; 

(b) $\forall t \in [0, T]$ $\mathcal A(t) \in \mL(C^2(\Omega), C(\Omega)) \cap \mL(C^{2+\beta}(\Omega), C^\beta(\Omega))$ and, for certain $\delta$ and $M$ in $\R^+$, $\forall  t \in [0, T]$, $\forall u \in C^{2+\beta}(\Omega))$, 
$$
\begin{array}{c}
\|\mathcal A(t)u\|_{C(\Omega))} \leq  \delta \|u\|_{C^{2}(\Omega)} + M \|u\|_{C^{1}(\Omega)}, \\ \\
 \|Ê\mathcal A\|_{C^{\beta/2}((0, T); \mL(C^2(\Omega), C(\Omega)))} \leq M, \\ \\
\|\mathcal A(t)u\|_{C^\beta(\Omega)} \leq \delta \|u\|_{C^{2+\beta}(\Omega)} + M \|u\|_{C^{2}(\Omega)};
\end{array}
$$

(c) $\forall t \in [0, T]$ $\mathcal B(t) \in \mL(C^1(\Omega), C(\partial \Omega)) \cap \mL(C^{2+\beta}(\Omega), C^{1+\beta}(\partial \Omega))$ and, $\forall  t \in [0, T]$, $\forall u \in C^{2+\beta}(\Omega)$, 
$$
\begin{array}{c}
\|\mathcal B\|_{C^{\beta/2}((0, T); \mL(C^1(\Omega), C(\partial \Omega)))} \leq M, \\ \\
\|\mathcal B(t)u\|_{C^{1+\beta}(\partial \Omega)} \leq \delta \|u\|_{C^{2+\beta}(\Omega)} + M \|u\|_{C^{2}(\Omega)}.
\end{array}
$$
Consider the system (\ref{eq4.4}). Then:

(I) there exists $\delta_0 \in \R^+$, depending only on $A$ and $B$, such that, if $\delta \leq \delta_0$, $f \in C^{\beta/2,\beta}((0, T) \times \Omega)$,  $h \in C^{\beta/2,1+\beta}((0, T) \times \partial \Omega)$, $u_0 \in C^{2+\beta}(\Omega)$ and
$$
A(\xi',D_\xi) u_0(\xi') + \mathcal A(0) u_0(\xi') + f(0,\xi') = -B(\xi',D_\xi) u_0(\xi') - \mathcal B(0)u_0(\xi') + h(0,\xi') \quad \forall \xi' \in \partial \Omega, 
$$
(\ref{eq4.4}) has a unique solution $u$ in $C^{1+\beta/2,2+\beta}((0, T) \times \Omega)$, with $D_t u_{|(0, T) \times \partial \Omega} \in B((0, T); C^{1+\beta}(\partial \Omega))$.

Let $T_0 \in \R^+$. Assume that (b) and (c) are satisfied and (I) holds, replacing $T$ with $T_0$. Take $0 < T \leq T_0$. Then: 

(II) there exists $C(T_0,A,B,\delta,M)$ in $\R^+$, such that
$$
\begin{array}{c}
\|u\|_{C^{1+\beta/2,2+\beta}((0, T) \times \Omega))}Ê+ \|D_tu_{|(0, T) \times \partial \Omega}\|_{B((0, T); C^{1+\beta}(\partial{\Omega}))}Ê+ \|u\|_{C^{\beta/2}((0, T); C^2(\Omega))} + \|u\|_{C^{\frac{1+\beta}{2}}((0, T); C^1(\Omega))}Ê\\ \\
\leq C(T_0,A,B,\delta,M) (\|f\|_{C^{\beta/2,\beta}((0, T) \times \Omega)} + \|u_0\|_{C^{2+\beta}(\Omega)}Ê+ \|h\|_{C^{\beta/2,1+\beta}((0, T) \times \partial \Omega)}). 
\end{array}
$$
(III) Suppose that $u_0 = 0$. Then, if $0 \leq \theta \leq 1$, 
\begin{equation}\label{eq4.2}
\|u\|_{C^\theta((0, T); C(\Omega))}Ê\leq C(T_0,A,B,\delta,M) T^{1-\theta}Ê(\|f\|_{C^{\beta/2,\beta}((0, T) \times \Omega)} + \|h\|_{C^{\beta/2,1+\beta}((0, T) \times \partial{\Omega})}),
\end{equation}
if $0 \leq \theta \leq 2+\beta$, 
\begin{equation}\label{eq4.3}
\|u\|_{B((0, T); C^\theta(\Omega))}Ê\leq C(T_0,A,B,\delta,M, \theta) T^{\frac{2+\beta-\theta}{2+\beta}} (\|f\|_{C^{\beta/2,\beta}((0, T) \times \Omega)} + \|h\|_{C^{\beta/2,1+\beta}((0, T) \times \partial \Omega)}). 
\end{equation}
(IV) Again supposing $u_0 = 0$, we have 
$$
\|u\|_{C^{\beta/2}((0, T); C^1(\Omega))}Ê\leq C(T_0,A,B,\delta,M) T^{1/2}Ê(\|f\|_{C^{\beta/2,\beta}((0, T) \times \Omega)} + \|h\|_{C^{\beta/2,1+\beta}((0, T) \times \partial \Omega)}).
$$

\end{lemma}

{\it Proof } We prove the result in several steps. 

{\it Step 1. We consider the case $u_0 = 0$ and prove an a priori estimate if $T$ is sufficiently small.  } 

So we have
$$
 f(0,\xi') =  h(0,\xi') \quad \forall \xi' \in \partial \Omega, 
$$
and we suppose that $u \in C^{1+\beta/2,2+\beta}((0, T) \times \Omega)$, with $D_t u_{|(0, T) \times \partial \Omega} \in B((0, T); C^{1+\beta}(\partial \Omega))$, solves (\ref{eq4.4}), with $u_0 = 0$. It is easily seen that, if we set
$$
F(t,\xi):= \mathcal A(t) u(t,\cdot)(\xi) + f(t,\xi), \quad (t,\xi) \in (0, T) \times \Omega,
$$
and
$$
H(t,\xi'):= -\mathcal B(t) u(t,\cdot)(\xi') + h(t,\xi'), \quad (t,\xi') \in (0, T) \times \partial \Omega,
$$
we have that $F \in C^{\beta/2,\beta}((0, T) \times \Omega)$ and $H \in C^{\beta/2,1+\beta}((0, T) \times \partial \Omega)$. So, by Lemma \ref{le4.1}, if (say) $T \leq 1$, we have
\begin{equation}\label{eq4.7}
\begin{array}{c}
\|u\|_{C^{1+\beta/2,2+\beta}((0, T) \times \Omega))}Ê+ \|D_tu_{|(0, T) \times \partial \Omega}\|_{B((0, T); C^{1+\beta}(\partial{\Omega}))} + \|u\|_{C^{\beta/2}((0, T); C^2(\Omega))} + \|u\|_{C^{\frac{1+\beta}{2}}((0, T); C^1(\Omega))}ÊÊ\\ \\
+ T^{-\frac{\beta}{2+\beta}} \|u\|_{B((0, T); C^2(\Omega))} + T^{-1/2}Ê \|u\|_{C^{\beta/2}((0, T); C^1(\Omega))} \\ \\
 \leq C(A,B)(\|f\|_{C^{\beta/2,\beta}((0, T) \times \Omega))} + \|h\|_{C^{\beta/2,1+\beta}((0, T) \times \partial \Omega))} \\ \\
+ \|\mathcal Au\|_{C^{\beta/2,\beta}((0, T) \times \Omega))} + \|\mathcal Bu\|_{C^{\beta/2,1+\beta}((0, T) \times \partial \Omega))}).
\end{array}
\end{equation}
If $0 \leq s < t \leq T$, we have
$$
\begin{array}{c}
\|\mathcal A(t)u(t) - \mathcal A(s)u(s)\|_{C( \Omega))} \\ \\
\leq \|[\mathcal A(t) - \mathcal A(s)]u(t)\|_{C( \Omega))} + \|\mathcal A(s)(u(t) - u(s))\|_{C( \Omega))} \\ \\
\leq M (t-s)^{\beta/2}Ê\|u(t)\|_{C^2(\Omega)} + \delta \|u(t) - u(s)\|_{C^2( \Omega))} + M \|u(t) - u(s)\|_{C^1( \Omega))}
\end{array}
$$
so that
$$
[\mathcal A u]_{C^{\beta/2}((0, T); C(\Omega))} \leq M \|u\|_{B((0, T); C^2(\Omega))} + \delta \|u\|_{C^{\beta/2}((0, T); C^2(\Omega))}+ M \|u\|_{C^{\beta/2}((0, T); C^1(\Omega))}
$$
and
\begin{equation}
\begin{array}{c}
\|\mathcal Au\|_{C^{\beta/2,\beta}((0, T) \times \Omega))} = \|\mathcal Au\|_{B((0, T); C^\beta(\Omega))} + [\mathcal A u]_{C^{\beta/2}((0, T); C(\Omega))} \\ \\
\leq \delta (\|u\|_{B((0, T); C^{2+\beta}(\Omega)} +  \|u\|_{C^{\beta/2}((0, T); C^2(\Omega))}) Ê+ 2M \|u\|_{B((0, T); C^{2}(\Omega))} + M \|u\|_{C^{\beta/2}((0, T); C^1(\Omega))}.
\end{array}
\end{equation}
Analogously, we can deduce, from (c), 
\begin{equation}\label{eq4.9}
\begin{array}{c}
\|\mathcal Bu\|_{C^{\beta/2,1+\beta}((0, T) \times \partial \Omega))} = \|\mathcal Bu\|_{B((0, T); C^{1+\beta}(\partial \Omega))} + [\mathcal B u]_{C^{\beta/2}((0, T); C(\partial \Omega))} \\ \\
\leq \delta \|u\|_{B((0, T); C^{2+\beta}(\Omega))} + M(  \|u\|_{B((0, T); C^{2}(\Omega))} + \|u\|_{C^{\beta/2}((0, T); C^1(\Omega))})
\end{array}
\end{equation}
and, from (\ref{eq4.7})-(\ref{eq4.9}), 
$$
\begin{array}{c}
\|u\|_{C^{1+\beta/2,2+\beta}((0, T) \times \Omega))}Ê+ \|D_tu_{|(0, T) \times \partial \Omega}\|_{B((0, T); C^{1+\beta}(\partial{\Omega}))} + \|u\|_{C^{\beta/2}((0, T); C^2(\Omega))}Ê\\ \\
+ T^{-\frac{\beta}{2+\beta}} \|u\|_{B((0, T); C^2(\Omega))} + T^{-1/2}Ê \|u\|_{C^{\beta/2}((0, T); C^1(\Omega))} \\ \\
 \leq C(A,B)[\|f\|_{C^{\beta/2,\beta}((0, T) \times \Omega))} + \|h\|_{C^{\beta/2,1+\beta}((0, T) \times \partial \Omega))} \\ \\
+ 2\delta (\|u\|_{B((0, T); C^{2+\beta}(\Omega))} +  \|u\|_{C^{\beta/2}((0, T); C^2(\Omega))}) Ê+ 3M \|u\|_{B((0, T); C^{2}(\Omega))} \\ \\
+ 2M \|u\|_{C^{\beta/2}((0, T); C^1(\Omega))}].
\end{array}
$$
Now we take $\delta \leq \delta_0$ in $\R^+$ and $T$ so small that  
\begin{equation}\label{eq4.10}
2C(A,B) \delta_0 \leq \frac{1}{2}, \quad 3MC(A,B) \leq \frac{1}{2} T^{-\frac{\beta}{2+\beta}}, \quad 2MC(A,B) \leq \frac{1}{2} T^{-1/2}. 
\end{equation}
We deduce the a priori estimate
\begin{equation}\label{eq4.11}
\begin{array}{c}
\|u\|_{C^{1+\beta/2,2+\beta}((0, T) \times \Omega))}Ê+ \|D_tu_{|(0, T) \times \partial \Omega}\|_{B((0, T); C^{1+\beta}(\partial{\Omega}))} + \|u\|_{C^{\beta/2}((0, T); C^2(\Omega))} + \|u\|_{C^{\frac{1+\beta}{2}}((0, T); C^1(\Omega))}ÊÊ\\ \\
+ T^{-\frac{\beta}{2+\beta}} \|u\|_{B((0, T); C^2(\Omega))} + T^{-1/2}Ê \|u\|_{C^{\beta/2}((0, T); C^1(\Omega))} \\ \\
 \leq 2 C(A,B)[\|f\|_{C^{\beta/2,\beta}((0, T) \times \Omega))} + \|h\|_{C^{\beta/2,1+\beta}((0, T) \times \partial \Omega))}].
 \end{array}
\end{equation}

{\it Step 2. We show the existence of a unique solution if $u_0 = 0$ and (\ref{eq4.10}) holds. 
}

To this aim, we consider, for every $\rho \in [0, 1]$, the system
\begin{equation}\label{eq4.12}
\left\{\begin{array}{l}
D_t u(t,\xi) - A(\xi,D_\xi) u(t,\xi) - \rho \mathcal A(t) u(t,\cdot)(\xi) = f(t,\xi), \quad t \in (0, T), \xi \in \Omega, \\ \\
D_t u(t,\xi') + B(\xi',D_\xi) u(t,\xi') + \rho \mathcal B(t)u(t,\cdot)(\xi') =  h(t,\xi'), \quad t \in (0, T),  \xi' \in \partial \Omega, \\ \\
u(0, \xi) = u_0(\xi), \quad \xi \in \Omega,
\end{array}
\right. 
\end{equation}
We set
$$
X:= \{u \in C^{1+\beta/2,2+\beta}((0, T) \times \Omega): D_tu_{|(0, T) \times \partial \Omega} \in B((0, T); C^{1+\beta}(\partial \Omega)), Êu(0,\cdot) = 0\}, 
$$
$$
Y:= \{(f,h) \in C^{\beta/2,\beta}((0, T) \times \Omega) \times C^{\beta/2,1+\beta}((0, T) \times \partial \Omega): f(0,\cdot)_{|\partial \Omega} = h(0,\cdot)\}.
$$
$X$ and $Y$ are Banach spaces with natural norms. For each $\rho$ we define the following operator $T_\rho$: 
$$
\left\{\begin{array}{l}
T_\rho : X \to Y, \\ \\
T_\rho u = (D_t u - A(\xi,D_\xi) u - \rho \mathcal Au, D_t u_{|(0, T) \times \partial \Omega} + B(\xi',D_\xi) u + \rho \mathcal B u).  
\end{array}
\right.
$$
Then $\rho \to T_\rho$ belongs to $C([0, 1]; \mL(X,Y))$ and $T_0$ is a linear and topological isomorphism between $X$ and $Y$ by Theorem \ref{th2.1}. Moreover, it is clear that for every $\rho \in [0, 1]$, (b) and (c) are satisfied if we replace $\mathcal A(t)$ with $\rho \mathcal A(t)$ and $\mathcal B(t)$ with $\rho \mathcal B(t)$ with the same constants $\delta$ and $M$. So, by the a priori estimate (\ref{eq4.11}) and the continuation principle, we deduce Step 2. 

{\it Step 3.  We prove (I), continuing to assume $\delta \leq \delta_0$ and $\delta_0$ and $T$ satisfying (\ref{eq4.10}).  }

We take $v(t,\xi):= u(t,\xi) - u_0(\xi)$ as new unknown. We obtain the system
\begin{equation}\label{eq4.13A}
\left\{\begin{array}{l}
D_t v(t,\xi) - A(\xi,D_\xi) v(t,\xi) - \mathcal A(t) v(t,\cdot)(\xi) = A(\xi,D_\xi) u_0(\xi) +  \mathcal A(t) u_0(\xi) + f(t,\xi), \quad t \in (0, T), \xi \in \Omega, \\ \\
D_t v(t,\xi') + B(\xi',D_\xi) v(t,\xi') + \mathcal B(t)v(t,\cdot)(\xi') = -B(\xi',D_\xi)u_0(\xi') -\mathcal B(t)u_0(\xi') + h(t,\xi'), \\ \\
 t \in (0, T), \quad   \xi' \in \partial \Omega, \\ \\
v(0, \xi) = 0, \quad \xi \in \Omega,
\end{array}
\right. 
\end{equation}
to which Step 2 is applicable. 

{\it Step 4. Proof of (I).}Ê

We begin by showing the uniqueness. So we suppose that $u \in C^{1+\beta/2,2+\beta}((0, T) \times \Omega)$, $D_t u_{|(0, T) \times \partial \Omega} \in B((0, T); C^{1+\beta}(\partial \Omega))$ and
$$
\left\{\begin{array}{l}
D_t u(t,\xi) - A(\xi,D_\xi) u(t,\xi) - \mathcal A(t) u(t,\cdot)(\xi) = 0, \quad t \in (0, T), \xi \in \Omega, \\ \\
D_t u(t,\xi') + B(\xi',D_\xi) u(t,\xi') + \mathcal B(t)u(t,\cdot)(\xi') = 0, \quad t \in (0, T),  \xi' \in \partial \Omega, \\ \\
u(0, \xi) = 0, \quad \xi \in \Omega,
\end{array}
\right. 
$$
We suppose, by contradiction, that $u \neq 0$. Let $\tau:= \inf \{t \in [0, T] : u(t,\cdot) \neq 0\}$. Then $0 \leq \tau < T$ and $u(\tau,\cdot) = 0$. We set
$$
u_1(t,\cdot):= u(\tau+t, \cdot), \quad t \in [0, T-\tau]. 
$$
Then, if $0 < T_1 \leq T-\tau$, $u_1$ solves the system
$$
\left\{\begin{array}{l}
D_t u_1(t,\xi) - A(\xi,D_\xi) u_1(t,\xi) - \mathcal A(\tau+t) u_1(t,\cdot)(\xi) = 0, \quad t \in (0, T_1), \xi \in \Omega, \\ \\
D_t u_1(t,\xi') + B(\xi',D_\xi) u_1(t,\xi') + \mathcal B(\tau+t)u_1(t,\cdot)(\xi') = 0, \quad t \in (0, T_1),  \xi' \in \partial \Omega, \\ \\
u_1(0, \xi) = 0, \quad \xi \in \Omega. 
\end{array}
\right. 
$$
Clearly, the families of operators $\{\mathcal A(\tau+t): t \in [0, T - \tau]\}$ and $\{\mathcal B(\tau+t): t \in [0, T - \tau]\}$ satisfy the conditions (b) and (c) with the same constants $\delta$ and $M$. So we deduce from Step 4 that, if we take $T_1$ sufficiently small, $u_{1|(0, T_1) \times \Omega}Ê= 0$, so that $u_{|(0, \tau + T_1) \times \Omega}Ê= 0$, in contradiction with the definition of $\tau$. 

We show the existence. We suppose that $T$ does not satisfy one of the majorities in (\ref{eq4.10}) and replace it with $\tau \in (0, T)$, in such a way that these majorities are satisfied by $\tau$. So we can apply Step 3 and get a unique solution
$u_1$ with domain $[0, \tau] \times \overline \Omega$. We observe that 
\begin{equation}\label{eq4.13}
\begin{array}{c}
A(\xi',D_\xi) u_1(\tau,\xi') + \mathcal A(\tau) u_1(\tau,\cdot)(\xi') + f(\tau,\xi') = D_tu_1(\tau, \xi') \\ \\
= -B(\xi',D_\xi) u_1(\tau,\xi') - \mathcal B(\tau)u_1(\tau,\cdot)(\xi') + h(\tau,\xi') \quad \forall \xi' \in \partial \Omega, 
\end{array}
\end{equation}
and consider the system
$$
\left\{\begin{array}{l}
D_t u_2(t,\xi) - A(\xi,D_\xi) u_2(t,\xi) - \mathcal A(\tau+t) u_2(t,\cdot)(\xi) = f(\tau+t,\xi), \quad t \in (0, \tau \wedge (T-\tau)), \xi \in \Omega, \\ \\
D_t u_2(t,\xi') + B(\xi',D_\xi) u_2(t,\xi') + \mathcal B(\tau+t)u_2(t,\cdot)(\xi') = h(\tau+t,\xi'), \quad t \in (0, \tau \wedge (T-\tau)),  \xi' \in \partial \Omega, \\ \\
u_2(0, \xi) = u_1(\tau,\xi), \quad \xi \in \Omega,
\end{array}
\right. 
$$
By (\ref{eq4.13}) Step 3 is applicable and it is easily seen that, if we set
$$
u(t,\xi) = \left\{\begin{array}{ll}
u_1(t,\xi) & \mbox{\rm if }Ê(t, \xi) \in [0, \tau] \times \overline \Omega, \\ \\
u_2(t-\tau,\xi) & \mbox{\rm if }Ê(t, \xi) \in [\tau, 2\tau \wedge T] \times \overline \Omega, 
\end{array}
\right.
$$
$u \in C^{1+\beta/2,2+\beta}((0, (2\tau) \wedge T) \times \Omega)$, with $D_t u_{|(0, (2\tau) \wedge T) \times \partial \Omega} \in B((0, (2\tau) \wedge T); C^{1+\beta}(\partial \Omega))$ and solves (\ref{eq4.4}) if we replace $T$ with 
$(2\tau) \wedge T$. In case $2\tau < T$, we can iterate the procedure and in a finite number of steps we construct a solution in $(0, T) \times \Omega$. 

II) If $T$ is so small that (\ref{eq4.10}) holds and $u_0 = 0$, we get the conclusion from the a priori estimate (\ref{eq4.11}). The case $u_0 \neq 0$, can be deduced applying the foregoing to the solution $v$ to (\ref{eq4.13A}). If $T$ does 
not satisfy (\ref{eq4.10}), we fix $T_1$ in $(0, T)$, satisfying it and obtain (II) applying the estimates in an interval of length, less or equal than $T_1$ $[T/T_1] + 1$ times. 

(III)-(IV) can be proved with the same arguments employed for the analogous estimates in Lemma \ref{le4.1}.

$\Box$

\begin{remark}\label{re4.3}
{\rm Lemma \ref{le4.1} is applicable in case
\begin{equation}
\mathcal A(t) u(\xi) = \sum_{|\alpha| \leq 2} r_\alpha(t,\xi) D_\xi^\alpha u(\xi), \quad (t,\xi) \in [0, T]Ê\times \overline \Omega,
\end{equation}
and
\begin{equation}
\mathcal B(t) u(\xi') = \sum_{|\alpha| \leq 1} \sigma_\alpha(t,\xi') D_x^\alpha u(\xi'), \quad (t,\xi) \in [0, T]Ê\times \partial \Omega,
\end{equation}
with $r_\alpha \in C^{\beta/2,\beta}((0, T) \times \Omega)$ ($|\alpha| \leq 2$) and $ \sigma_\alpha \in C^{\beta/2,1+\beta}((0, T) \times \partial \Omega)$ ($|\alpha| \leq 1$),
if 
\begin{equation}
\sum_{|\alpha| = 2} \|r_\alpha\|_{C((0, T) \times \Omega)} + \sum_{|\alpha| = 1} \|\sigma_\alpha\|_{C((0, T) \times \partial \Omega)} \leq d
\end{equation}
and 
\begin{equation}
\sum_{|\alpha| \leq 2} \|r_\alpha\|_{C^{\beta/2,\beta}((0, T) \times \Omega)} + \sum_{|\alpha| \leq 1} \|\sigma_\alpha\|_{C^{\beta/2,1+\beta}((0, T) \times \partial \Omega)} \leq N,
\end{equation}
with $d$ and $N$ in $\R^+$. Then we have
\begin{equation}
\|\mathcal A(t) u\|_{C(\Omega)} \leq C(n)(d \|u\|_{C^2(\Omega)} + N \|u\|_{C^1(\Omega)}),
\end{equation}

\begin{equation}
\|Ê\mathcal A\|_{C^{\beta/2}((0, T); \mL(C^2(\Omega), C(\Omega)))} \leq C(n)N, 
\end{equation}

\begin{equation}
\|\mathcal A(t)u\|_{C^\beta(\Omega)} \leq C(n)(\delta \|u\|_{C^{2+\beta}(\Omega)} + N \|u\|_{C^{2}(\Omega)}).
\end{equation}
Moreover, 

$$
\begin{array}{c}
\|\mathcal B\|_{C^{\beta/2}((0, T); \mL(C^1(\Omega), C(\partial \Omega)))} \leq C(n)N, \\ \\
\|\mathcal B(t)u\|_{C^{1+\beta}(\partial \Omega)} \leq C(n)(\delta \|u\|_{C^{2+\beta}(\Omega)} + N \|u\|_{C^{2}(\Omega)}).
\end{array}
$$
So we can take $\delta = C(n)d$ and $M = C(n)N$. 
}
\end{remark}

\begin{lemma}\label{le4.4A}
Let $\Omega$ be an open bounded subset in $\R^n$, lying on one side of its boundary $\partial \Omega$, which is a submanifold of class $C^{2+\beta}$ of $\R^n$. Let $N$ and $\nu$ belong to $\R^+$. We set
$$
\begin{array}{ll}
A(N,\nu)  := & \{((a_\alpha)_{|\alpha| \leq 2}, (b_\gamma)_{|\gamma| \leq 1}): \sum_{|\alpha| \leq 2} \|a_\alpha\|_{C^{\beta}(\Omega; \R)} +  \sum_{|\gamma| \leq 1} \|b_\gamma\|_{C^{1+\beta} (\partial \Omega; \R)} \leq N, \\ \\
& \sum_{|\alpha| = 2}Êa_\alpha(x) \xi^\alpha \geq \nu |\xi|^2 \quad  \forall (x,\xi) \in \overline \Omega \times \R^n, \sum_{\gamma = 1} b_\gamma(x')\nu_\gamma(x') \geq \nu \quad \forall (x',\xi) \in \partial \Omega\}.
\end{array}
$$
Then the constants $C(T_0,A,B)$ and $C(T_0,A,B,\theta)$ appearing in the statement of Lemma \ref{le4.1} can be taken independently of $A$ and $B$ if  $((a_\alpha)_{|\alpha| \leq 2}, (b_\gamma)_{|\gamma| \leq 1}) \in A(N,\nu)$. 
\end{lemma}

{\it Proof } Let $((a^0_\alpha)_{|\alpha| \leq 2}, (b_\gamma^0)_{|\gamma| \leq 1}) \in A(N,\nu)$. Then, by Remark \ref{re4.3}, there exists $d$ in $\R^+$, such that the constants $C(T_0,A,B)$ and $C(T_0,A,B,\theta)$ can be chosen
independently of  $((a_\alpha)_{|\alpha| \leq 2}, (b_\gamma)_{|\gamma| \leq 1}) \in A(N,\nu)$, in case 
$$\sum_{|\alpha| = 2}Ê\|a_\alpha - a_\alpha^0\|_{C(\Omega)} + \sum_{|\gamma| = 1}Ê\|b_\gamma - b_\gamma^0\|_{C(\Omega)} < d. $$
Now we observe that, by the theorem of Ascoli-Arzel$\grave {\rm a}$,  $A(N,\nu)$ is compact in $C(\Omega)^{n^2+n+1}Ê\times C(\partial \Omega)^{n+1}$ (we recall, that, for example, a bounded sequence in 
$C^{\beta/2,1+\beta}((0, T) \times \partial \Omega)$ contains a subsequence uniformly converging to an element of $C^{\beta/2,1+\beta}((0, T) \times \partial \Omega)$). So it can be covered by a finite number of balls of the described type.
The conclusion follows.

$\Box$

\begin{theorem}\label{th4.5}
Consider the system (\ref{eq4.1}), with the conditions (AF1)-(AF3). Then:

(I) if $f \in C^{\beta/2,\beta}((0, T) \times \Omega)$,  $h \in C^{\beta/2,1+\beta}((0, T) \times \partial \Omega)$, $u_0 \in C^{2+\beta}(\Omega)$ and
$$
A(0,x',D_x) u_0(x') + f(0,x') = -B(0,x',D_x) u_0(x') + h(0,x') \quad \forall x' \in \partial \Omega, 
$$
(\ref{eq4.1}) has a unique solution $u$ in $C^{1+\beta/2,2+\beta}((0, T) \times \Omega)$, with $D_t u_{|(0, T) \times \partial \Omega} \in B((0, T); C^{1+\beta}(\partial \Omega))$.

Let $T_0 \in \R^+$ and assume that $0 < T \leq T_0$. Then: 

(II) there exists $C(T_0,N,\nu)$ in $\R^+$, such that
$$
\begin{array}{c}
\|u\|_{C^{1+\beta/2,2+\beta}((0, T) \times \Omega))}Ê+ \|D_tu_{|(0, T) \times \partial \Omega}\|_{B((0, T); C^{1+\beta}(\partial{\Omega}))}Ê+ \|u\|_{C^{\beta/2}((0, T); C^2(\Omega))} + \|u\|_{C^{\frac{1+\beta}{2}}((0, T); C^1(\Omega))}Ê\\ \\
\leq C(T_0,N,\nu) (\|f\|_{C^{\beta/2,\beta}((0, T) \times \Omega)} + \|u_0\|_{C^{2+\beta}(\Omega)}Ê+ \|h\|_{C^{\beta/2,1+\beta}((0, T) \times \partial \Omega)}). 
\end{array}
$$
(III) Suppose that $u_0 = 0$. Then, if $0 \leq \theta \leq 1$, 
\begin{equation}\label{eq4.22}
\|u\|_{C^\theta((0, T); C(\Omega))}Ê\leq C(T_0,N,\nu) T^{1-\theta}Ê(\|f\|_{C^{\beta/2,\beta}((0, T) \times \Omega)} + \|h\|_{C^{\beta/2,1+\beta}((0, T) \times \partial{\Omega})}),
\end{equation}
if $0 \leq \theta \leq 2+\beta$, 
\begin{equation}\label{eq4.23}
\|u\|_{B((0, T); C^\theta(\Omega))}Ê\leq C(T_0,N,\nu, \theta) T^{\frac{2+\beta-\theta}{2+\beta}} (\|f\|_{C^{\beta/2,\beta}((0, T) \times \Omega)} + \|h\|_{C^{\beta/2,1+\beta}((0, T) \times \partial \Omega)}). 
\end{equation}
(IV) Again supposing $u_0 = 0$, we have 
$$
\|u\|_{C^{\beta/2}((0, T); C^1(\Omega))}Ê\leq C(T_0,N,\nu) T^{1/2}Ê(\|f\|_{C^{\beta/2,\beta}((0, T) \times \Omega)} + \|h\|_{C^{\beta/2,1+\beta}((0, T) \times \partial \Omega)}).
$$
\end{theorem}

{\it Proof  }ÊLet $t_0 \in [0, T)$. We consider the system
\begin{equation}\label{eq4.24}
\left\{\begin{array}{c}
D_t v(t,x) - A(t_0+t,x,D_x) v(t,x)  \\ \\ = D_t v(t,x) - A(t_0,x,D_x) v(t,x) - [A(t_0+t,x,D_x) - A(t_0,x,D_x)] v(t,x) + \phi(t,x), \\ \\ t \in (0, \tau), x \in \Omega, \\ \\
D_t v(t,x') + B(t_0+t,x',D_x) v(t,x') \\ \\
= D_t v(t,x') + B(t_0,x',D_x) v(t,x') + [B(t_0+t,x',D_x) - B(t_0,x',D_x)] v(t,x') = \psi(t,x'), \\ \\
 t \in (0, \tau),  x' \in \partial \Omega, \\ \\
v(0, x) = v_0(x), \quad x \in \Omega,
\end{array}
\right. 
\end{equation}
We have, for $|\alpha| = 2$, $t \in [0, T-t_0]$, $x \in \Omega$, $|\gamma| = 1$, $x' \in \partial \Omega$, 
$$
|a_\alpha(t_0+t,x) - a_\alpha(t_0,x)| \leq N t^{\beta/2}, |b_\gamma(t_0+t,x') - b_\gamma(t_0,x')| \leq N t^{\beta/2},
$$
So, by Lemma \ref{le4.2}, Remark  \ref{re4.3} and Lemma \ref{le4.2}, there exists $\tau_0 \in \R^+$, independent of $t_0$ in $[0, T)$, such that, if $0 < \tau \leq \tau_0 \wedge (T-\tau_0)$, $\phi \in C^{\beta/2,\beta}((0, \tau) \times \Omega))$, $\psi \in C^{\beta/2,1+\beta}((0, \tau) \times \partial \Omega)$, $v_0 \in C^{2+\beta}(\Omega)$, 
$$
A(t_0,x',D_x) v_0(x') + \phi(0,x') = - B(t_0,x',D_x) v_0(x') + \psi(0,x') \quad \forall x' \in \partial \Omega,
$$
(\ref{eq4.24}) has a unique solution $v$ in $C^{1+\beta/2,2+\beta}((0, \tau) \times \Omega)$, with $D_t v_{|(0, \tau) \times \partial \Omega} \in B((0, \tau); C^{1+\beta}(\partial \Omega))$. Moreover, 
 there exists $C(\tau_0,N,\nu)$ in $\R^+$, such that
$$
\begin{array}{c}
\|v\|_{C^{1+\beta/2,2+\beta}((0, \tau) \times \Omega))}Ê+ \|D_tv_{|(0, \tau) \times \partial \Omega}\|_{B((0, \tau); C^{1+\beta}(\partial{\Omega}))}Ê+ \|v\|_{C^{\beta/2}((0, \tau); C^2(\Omega))} + \|v\|_{C^{\frac{1+\beta}{2}}((0, \tau); C^1(\Omega))}Ê\\ \\
\leq C(\tau_0,N,\nu) (\|\phi\|_{C^{\beta/2,\beta}((0, \tau) \times \Omega)} + \|v_0\|_{C^{2+\beta}(\Omega)}Ê+ \|\psi\|_{C^{\beta/2,1+\beta}((0, \tau) \times \partial \Omega)}). 
\end{array}
$$
So we can begin by constructing a solution $v$ in $(0, \tau_0 \wedge T) \times \Omega$. In case $\tau_0 < T$, we can consider (\ref{eq4.24}) with $t_0 = \tau_0$, $\phi(t,x) = f(\tau+t,x)$, $\psi(t,x') = h(\tau+t,x')$, $v_0 = u(\tau_0,\cdot)$.
In a finite number of steps we get (I) and (II). (III) and (IV) can be obtained as in the proof of Lemma \ref{le4.1}.

$\Box$

\section{Quasilinear problems}

\setcounter{equation}{0}

We are going to discuss system (\ref{eq5.1}).

We have:

\begin{lemma}\label{le5.1}
Let $\Omega$ be an open bounded subset of $\R^n$ and let $a : [0, T] \times \overline \Omega \times \R^{n+1} \to \R$. We assume the following:

(a) $\forall (t,x) \in [0, T]Ê\times \overline \Omega$ $a(t,x,\cdot,\cdot) \in \k^1(\R^{n+1})$; 

(b) for some $\beta \in (0, 1)$,  $\forall R \in \R^+$ there exists $C(R) \in \R^+$ such that, $\forall s, t \in [0, T]$, $\forall x, y \in \overline \Omega$, $\forall u,v \in \R$, $\forall p, q \in \R^n$, 
$$
\begin{array}{c}
|a(t,x,u,p) - a(s,y,v,q)| + |\nabla_{u,p} a(t,x,u,p) - \nabla_{u,p}a(s,y,v,q)| \\ \\
 \leq C(R)(|t-s|^{\beta/2} + |x-y|^\beta + |u-v| + |p-q|)
\end{array}
$$
whenever $|u| + |v| + |p| + |q| \leq R$. For $\tau \in (0, T]$ and $u: [0, \tau] \times \overline \Omega \to \R$, define 
$$
A(u)(t,x):= a(t,x,u(t,x),\nabla_x u(t,x))
$$
whenever the second term has a meaning. Then:

(I) if $u \in C^{\beta/2}((0, \tau), C^1(\Omega)) \cap B((0, \tau); C^{1+\beta}(\Omega))$, $R \in \R^+$  and 
$$\max\{\|u\|_{C^{\beta/2}((0, \tau), C^1(\Omega))}, \|u\|_{B((0, \tau), C^{1+\beta}(\Omega))}\} \leq R, $$
$A(u) \in C^{\beta/2,\beta}((0, \tau) \times \Omega)$ and $\|A(u)\|_{C^{\beta/2,\beta}((0, \tau) \times \Omega)} \leq C_1(R)$; 

(II) if $u,v \in C^{\beta/2}((0, \tau), C^1(\Omega)) \cap B((0, \tau); C^{1+\beta}(\Omega))$  and 
$$\max\{\|u\|_{C^{\beta/2}((0, \tau), C^1(\Omega))}, \|u\|_{B((0, \tau), C^{1+\beta}(\Omega))}, \|v\|_{C^{\beta/2}((0, \tau), C^1(\Omega))}, \|v\|_{B((0, \tau), C^{1+\beta}(\Omega))}\} \leq R, $$
then
$$\|A(u) - A(v)\|_{C^{\beta/2,\beta}((0, \tau) \times \Omega)} \leq C_2(R) (\|u-v\|_{C^{\beta/2}((0, \tau), C^1(\Omega))} + \|u-v\|_{B((0, \tau), C^{1+\beta}(\Omega))}).$$
\end{lemma}

\begin{proof} (I) If $(t,x) \in (0, T) \times \Omega$,
$$
\begin{array}{c}
|A(u)(t,x)| \leq |a(t,x,u(t,x),\nabla_xu(t,x)) - a(t,x,0,0)| + |a(t,x,0,0)| \\ \\
 \leq \|a(\cdot,\cdot,0,0)\|_{C((0, T) \times \Omega)} + C(R) (|u(t,x)| + |\nabla_x u(t,x)|) \leq C_3(R). 
\end{array}
$$
Moreover, if $s, t \in (0, \tau)$ and $x \in \Omega$, 
$$
\begin{array}{c}
|a(t,x,u(t,x), \nabla_xu(t,x)) - a(s,x,u(s,x), \nabla_xu(s,x))| \\ \\
\leq C(R)(|t-s|^{\beta/2} + |u(t,x) - u(s,x)| + |\nabla_xu(t,x) - \nabla_xu(s,x)|) \\ \\
 \leq C(R)(|t-s|^{\beta/2} + (t-s)^{\beta/2} \|u\|_{C^{\beta/2}((0, T); C(\Omega))}  + (t-s)^{\beta/2}  \|u\|_{C^{\beta/2}((0, T); C^1(\Omega))}) \\ \\
 \leq C_4(R) |t-s|^{\beta/2}. 
\end{array}
$$
Analogously, one can show that $\|A(u)\|_{B((0, \tau); C^{\beta}(\Omega))} \leq C_5(R)$. 

(II) If $(t,x) \in (0, \tau) \times \Omega$, 
$$
\begin{array}{c}
|A(u)(t,x) - A(v)(t,x)| \leq C(R)(|u(t,x) - v(t,x)| + |\nabla_x u(t,x) - \nabla_x v(t,x)|) \\ \\
 \leq 2C(R) \|u - v\|_{C((0, \tau); C^1(\Omega))} \leq 2C(R) \|u - v\|_{C^{\beta/2}((0, \tau); C^1(\Omega))}. 
\end{array}
$$
If $t,s \in (0, \tau)$ and $x  \in \Omega$, 
$$
\begin{array}{c}
|(A(u)(t,x) - A(v)(t,x)) - (A(u)(s,x) - A(v)(s,x))| \\ \\
= |\int_0^1 \nabla_{u,p} a(t,x, v(t,x) + r (u(t,x) - v(t,x)), \nabla_xv(t,x) + r ( \nabla_xu(t,x) -  \nabla_xv(t,x))) dr \\ \\
\cdot (u(t,x) - v(t,x), \nabla_xu(t,x) -  \nabla_xv(t,x)) \\ \\
- \int_0^1 \nabla_{u,p} a(s,x, v(s,x) + r (u(s,x) - v(s,x)), \nabla_xv(s,x) + r ( \nabla_xu(s,x) -  \nabla_xv(s,x))) dr \\ \\
\cdot (u(s,x) - v(s,x), \nabla_xu(s,x) -  \nabla_xv(s,x))| \\ \\
\leq |\int_0^1 \nabla_{u,p} a(t,x, v(t,x) + r (u(t,x) - v(t,x)), \nabla_xv(t,x) + r ( \nabla_xu(t,x) -  \nabla_xv(t,x)) dr| \\ \\
\times |(u(t,x) - v(t,x) - (u(s,x) - v(s,x)), \nabla_xu(t,x) -  \nabla_xv(t,x) - (\nabla_xu(s,x) - \nabla_xv(s,x)))| \\ \\
+  \int_0^1 |\nabla_{u,p} a(t,x, v(t,x) + r (u(t,x) - v(t,x)), \nabla_xv(t,x) + r ( \nabla_xu(t,x) -  \nabla_xv(t,x))  \\ \\
- \nabla_{u,p} a(s,x, v(s,x) + r (u(s,x) - v(s,x)), \nabla_xv(s,x) + r ( \nabla_xu(s,x) -  \nabla_xv(s,x)) | dr \\ \\
\times |(u(s,x) - v(s,x), \nabla_xu(s,x) -  \nabla_xv(s,x))| = I + J
\end{array}
$$
and 
$$
\begin{array}{c}
I \leq C_5(R)(|u(t,x) - v(t,x) - (u(s,x) - v(s,x))| \\ \\
+ |\nabla_xu(t,x) -  \nabla_xv(t,x) - (\nabla_xu(s,x) - \nabla_xv(s,x))|) \\ \\
\leq 2C_5(R) |t-s|^{\beta/2}Ê\|u-v\|_{C^{\beta/2}((0, \tau); C^1(\Omega))},
\end{array}
$$
$$
\begin{array}{c}
J \leq C_6(R) (|t-s|^{\beta/2} + |u(t,x) - u(s,x)| + |v(t,x) - v(s,x)| \\ \\
+ |\nabla_xu(t,x) - \nabla_xu(s,x)| + |\nabla_x v(t,x) - \nabla_x v(s,x)|) \\ \\
\times (|u(s,x) - v(s,x)| + |\nabla_xu(s,x) -  \nabla_xv(s,x)|) \\ \\
\leq C_7(R) |t-s|^{\beta/2} (1 + \|u\|_{C^{\beta/2}((0, \tau); C^1(\Omega))} +  \|v\|_{C^{\beta/2}((0, \tau); C^1(\Omega))})\|u-v\|_{C((0, \tau); C^1(\Omega))}. 
\end{array}
$$
So 
$$[A(u) - A(v)]_{C^{\beta/2}((0, \tau); C(\Omega))} \leq C_8(R) \|u-v\|_{C^{\beta/2}((0, \tau); C^1(\Omega))}. $$
Analogously, one can show that
$$\|A(u) - A(v)\|_{B((0, \tau); C^\beta(\Omega))} \leq C_9(R) \|u-v\|_{B((0, \tau); C^{1+\beta}(\Omega))}. $$
\end{proof}

\begin{lemma}\label{le5.2}
Let $\Omega$ be an open bounded subset of $\R^n$, $\beta \in (0, 1)$) and let $b : [0, T] \times \overline \Omega \times \R \to \R$. We assume the following: 

(a) $b$ is continuous together with its partial derivatives $D_{x_j} b$, $D_ub$, $D^2_{x_ju} b$, $D^2_{u} b$ ($1 \leq j \leq n$); 

(b) $\forall R \in \R^+$ there exists $C(R) \in \R^+$ such that, $\forall t,s \in [0, T]$, $\forall x \in \overline \Omega$, $\forall u,v \in \R$ with $|u| + |v| \leq R$, 
$$
|b(t,x,u) - b(s,x,v)| + |D_ub(t,x,u) - D_ub(s,x,v)| \leq C(R) (|t-s|^{\beta/2} + |u-v|). 
$$
(c) $\forall R \in \R^+$ there exists $C(R) \in \R^+$ such that, $\forall t \in [0, T]$, $\forall x,y \in \overline \Omega$, $\forall u,v \in \R$ with $|u| + |v| \leq R$, 
$$
|\nabla_{x,u} b(t,x,u) - \nabla_{x,u} b(t,y,v)| + |D_u\nabla_{x,u} b(t,x,u) - D_u\nabla_{x,u} b(t,y,v)| \leq C(R) (|x-y|^{\beta} + |u-v|). 
$$
For $\tau \in (0, T]$ and $u: [0, \tau] \times \overline \Omega \to \R$, define 
$$
B(u)(t,x):= b(t,x,u(t,x)).
$$
Then:

(I) if $u \in C^{\beta/2,1+\beta}((0, \tau) \times \Omega))$, $R \in \R^+$  and 
$$\|u\|_{C^{\beta/2,1+\beta}((0, \tau) \times \Omega)} \leq R, $$
$B(u) \in C^{\beta/2,1+\beta}((0, \tau) \times \Omega)$ and $\|B(u)\|_{C^{\beta/2,1+\beta}((0, \tau) \times \Omega)} \leq C_1(R)$; 

(II) if $u,v \in C^{\beta/2,1+\beta}((0, \tau) \times \Omega)$  and 
$$\max\{\|u\|_{C^{\beta/2,1+\beta}((0, \tau) \times \Omega)}, \|v\|_{C^{\beta/2,1+\beta}((0, \tau) \times \Omega)}\} \leq R, $$
then
$$\|B(u) - B(v)\|_{C^{\beta/2,1+\beta}((0, \tau) \times \Omega)} \leq C_2(R) \|u-v\|_{C^{\beta/2,1+\beta}((0, \tau) \times \Omega))}.$$
\end{lemma}

\begin{proof} It follows the lines of the proof of Lemma \ref{le5.1}; of course, one should use the elementary formula
$$
D_{x_j}B(u)(t,x) = D_{x_j} b(t,x,u(t,x)) + D_{u} b(t,x,u(t,x)) D_{x_j} u(t,x) \quad (1 \leq j \leq n). 
$$
\end{proof}

Now we are able to study system (\ref{eq5.1}). We introduce the following assumptions: 

\medskip

{\it (AG1) $\beta \in (0, 1)$, $\Omega$ is an open bounded subset of $\R^n$, with boundary $\partial \Omega$ of class $C^{2+\beta}$; 

\medskip

(AG2) $T \in \R^+$ and $\forall \alpha \in \N_0^n$, with $|\alpha| = 2$, $a_\alpha, f : [0, T] \times \overline \Omega \times \R^{n+1} \to \R$;
 $\forall (t,x) \in [0, T]Ê\times \overline \Omega$ $a_\alpha(t,x,\cdot,\cdot), f(t,x,\cdot,\cdot) \in \k^1(\R^{n+1})$; $\forall R \in \R^+$ there exists $C(R) \in \R^+$ such that, $\forall s, t \in [0, T]$, $\forall x, y \in \overline \Omega$, $\forall u,v \in \R$, $\forall p, q \in \R^n$, 
$$
\begin{array}{c}
\sum_{|\alpha| = 2} (|a_\alpha(t,x,u,p) - a_\alpha(s,y,v,q)| + |\nabla_{u,p} a_\alpha(t,x,u,p) - \nabla_{u,p}a_\alpha(s,y,v,q)|) \\ \\
+ |f(t,x,u,p) - f(s,y,v,q)| + |\nabla_{u,p} f(t,x,u,p) - \nabla_{u,p} f(s,y,v,q)| \\ \\
 \leq C(R)(|t-s|^{\beta/2} + |x-y|^\beta + |u-v| + |p-q|). 
\end{array}
$$

\medskip

(AG3) $\forall j  \in \{1, \dots, n\}$ $b_j, h : [0, T] \times \overline \Omega \times \R \to \R$; they are continuous, together with their  derivatives $D_{x_i} b_j$, $D_ub_j$, $D^2_{x_iu} b_j$, $D^2_{u} b_j$,
$D_{x_i} h$, $D_uh$, $D^2_{x_iu} h$, $D^2_{u} h$ ($1 \leq i \leq n$); $\forall R \in \R^+$ there exists $C(R) \in \R^+$ such that, $\forall t,s \in [0, T]$, $\forall x \in \overline \Omega$, $\forall u, v \in \R$ with 
$|u| + |v| \leq R$, 
$$
\begin{array}{c}
\sum_{j=1}^n (|b_j(t,x,u) - b_j(s,x,v)|  \\ \\
+ |D_u b_j(t,x,u) - D_ub_j(s,x,v)|) + |h(t,x,u) - h(s,x,v)| + |D_u h(t,x,u) - D_uh(s,x,v)| \\ \\
\leq C(R)(|t - s|^{\beta/2} + |u-v|); 
\end{array}
$$
 $\forall R \in \R^+$ there exists $C(R) \in \R^+$ such that, $\forall t \in [0, T]$, $\forall x,y \in \overline \Omega$, $\forall u,v \in \R$ with $|u| + |v| \leq R$, 
$$
\begin{array}{c}
\sum_{j=1}^n (|\nabla_{x,u} b_j(t,x,u) - \nabla_{x,u} b_j(t,y,v)| + |D_u\nabla_{x,u} b_j(t,x,u) - D_u\nabla_{x,u} b_j(t,y,v)|) \\ \\
 + |\nabla_{x,u} h(t,x,u) - \nabla_{x,u} h(t,y,v)| + |D_u\nabla_{x,u} h(t,x,u) - D_u\nabla_{x,u} h(t,y,v)| \leq C(R) (|x-y|^{\beta} + |u-v|). 
\end{array}
$$

\medskip

(AG4) $\forall (t,x,u,p) \in [0, T] \times \overline \Omega \times \R^{n+1}$ there exists $\nu(t,x,u,p) \in \R^+$ such that
$$\sum_{|\alpha| = 2} a_\alpha(t,x,u,p) \xi^\alpha \geq \nu(t,x,u,p) |\xi|^2 \quad \forall \xi \in \R^n;$$
$\forall (t,x',u) \in [0, T] \times \partial \Omega \times \R$ there exists $\nu(t,x',u) \in \R^+$ such that 
$$\sum_{j=1}^n b_j(t,x',u) \nu_j(x')  \geq \nu(t,x',u).$$
}

\medskip

\begin{lemma}
Assume that (AG1)-(AG4) hold. Let $u_0 \in C^{2+\beta}(\Omega)$ be such that 
$$
\begin{array}{c}
\sum_{|\alpha| = 2} a_\alpha(0,x', u_0(x'), \nabla_x u_0(x')) D_x^\alpha u_0(x') + f(0,x', u_0(x'), \nabla_x u_0(x')) \\ \\
=  -\sum_{j=1}^n b_j(0,x', u_0(x')) D_{x_j} u_0(x') + h(0,x', u_0(x')), \quad \forall x' \in \partial \Omega. 
\end{array}
$$
Consider the system
\begin{equation}\label{eq5.2}
\left\{\begin{array}{ll}
D_t U_0(t,x) = \sum_{|\alpha| = 2}Êa_\alpha(t,x, u_0(x), \nabla_x u_0(x)) D_x^\alpha U_0(t,x) + f(t,x, u_0(x), \nabla_x u_0(x)), & t \in (0, T), x \in \Omega, \\ \\
D_t U_0(t,x') + \sum_{j=1}^nÊb_j(t,x', u_0(x')) D_{x_j} U_0(t,x') = h(t,x', u_0(x')), & t \in (0, T),  x' \in \partial \Omega, \\ \\
U_0(0,x) = u_0(x), & x \in \Omega. 
\end{array}
\right.
\end{equation}
Then (\ref{eq5.2}) has a unique solution $U_0$ belonging to $C^{1+\beta/2,2+\beta}((0, T) \times \Omega)$, with $D_t U_{0|(0, T) \times \partial \Omega} \in C^{\beta/2,1+\beta}$ $((0, T) \times \partial \Omega)$.

\end{lemma}

\begin{proposition}\label{pr5.4}
Consider the system (\ref{eq5.1}), with the assumptions $(AG1)-(AG3)$. Let $u_0 \in C^{2+\beta}(\Omega)$ be such that 
$$
\begin{array}{c}
\sum_{|\alpha| = 2}Êa_\alpha(0,x', u_0(x'), \nabla_x u_0(x')) D_x^\alpha u_0(x') + f(0,x', u_0(x')) \\ \\
= -\sum_{j=1}^nÊb_j(0,x', u_0(x')) D_{x_j} u_0(x') + h(0,x', u_0(x')), \quad \forall x' \in \partial \Omega. 
\end{array}
$$
Let $R \in \R^+$. Then there exists $\tau(R) \in (0, T]$ such that, if $0 < \tau \leq \tau(R)$, (\ref{eq5.1}) has a unique solution $u$ in $C^{1+\beta/2, 2+\beta}((0, \tau) \times \Omega)$, with $D_tu_{|(0, \tau) \times \partial \Omega} \in B((0, \tau); C^{1+\beta}(\partial \Omega))$, satisfying
$$
\|u - U_0\|_{C^{\beta/2}((0, \tau); C^1(\Omega))}Ê+ \|u - U_0\|_{B((0, \tau); C^{1+\beta}(\Omega))} \leq R. 
$$

\end{proposition}

\begin{proof} We consider, for $\tau \in (0, T]$, $R \in \R$, the set 
$$
\begin{array}{ll}
\mM(\tau,R) := &  \{U \in C^{\beta/2}((0, \tau); C^1(\Omega)) \cap B((0, \tau); C^{1+\beta}(\Omega)): U(0,\cdot) = u_0, \\ \\
& \|u - U_0\|_{C^{\beta/2}((0, \tau); C^1(\Omega))}Ê+ \|u - U_0\|_{B((0, \tau); C^{1+\beta}(\Omega))} \leq R\}
\end{array}
$$
which is a complete metric space with the distance
$$
d(U,V):=  \|U - V\|_{C^{\beta/2}((0, \tau); C^1(\Omega))}Ê+ \|U - V\|_{B((0, \tau); C^{1+\beta}(\Omega))}. 
$$
If $U \in \mM(\tau,R)$, we consider the system
\begin{equation}\label{eq5.3}
\left\{\begin{array}{ll}
D_t u(t,x) = \sum_{|\alpha| = 2}Êa_\alpha(t,x, U(t,x), \nabla_x U(t,x)) D_x^\alpha u(t,x) + f(t,x, U(t,x), \nabla_x U(t,x)), & t \in (0, \tau), x \in \Omega, \\ \\
D_t u(t,x') + \sum_{j=1}^nÊb_j(t,x', U(t,x')) D_{x_j} u(t,x') = h(t,x, U(t,x)), & t \in (0, \tau),  x' \in \partial \Omega, \\ \\
u(0,x) = u_0(x), & x \in \Omega. 
\end{array}
\right.
\end{equation}
Then, applying Lemmata \ref{le5.1} and \ref{le5.2} and Theorem \ref{th4.5}, we obtain that (\ref{eq5.3}) has a unique solution $u = u(U)$ in $C^{1+\beta/2, 2+\beta}((0, \tau) \times \Omega)$, with $D_tu_{|(0, \tau) \times \partial \Omega}
 \in B((0, \tau); C^{1+\beta}(\partial \Omega))$. Moreover, there exists $C_1(R)$ in $\R^+$, independent of $\tau$ and $U$, such that 
 \begin{equation}\label{eq5.4}
\begin{array}{c}
\|u\|_{C^{1+\beta/2,2+\beta}((0, T) \times \Omega))}Ê+ \|D_tu_{|(0, T) \times \partial \Omega}\|_{B((0, T); C^{1+\beta}(\partial{\Omega}))}Ê+ \|u\|_{C^{\beta/2}((0, T); C^2(\Omega))} \\ \\+ \|u\|_{C^{\frac{1+\beta}{2}}((0, T); C^1(\Omega))}Ê
\leq C_1(R). 
\end{array}
\end{equation}
Hence, we have
$$
\left\{\begin{array}{cc}
D_t (u - U_0)(t,x) = \sum_{|\alpha| = 2}Êa_\alpha(t,x, u_0(x), \nabla_x u_0(x)) D_x^\alpha (u - U_0)(t,x) \\ \\
+ \sum_{|\alpha| = 2} [a_\alpha(t,x, U(t,x), \nabla_x U(t,x)) - a_\alpha(t,x, u_0(x), \nabla_x u_0(x))] D_x^\alpha u(t,x) \\ \\
 + f(t,x, U(t,x)) - f(t,x,u_0(x)),  & t \in (0, \tau), x \in \Omega, \\ \\
D_t (u - U_0)(t,x') +  \sum_{j=1}^nÊb_j(t,x', u_0(x')) D_{x_j} (u - U_0) (t,x')  \\ \\
+  \sum_{j=1}^nÊ[b_j(t,x', U(t,x')) - b_j(t,x', u_0(x'))] D_{x_j} u(t,x') \\ \\
= h(t,x', U(t,x')) - h(t,x, u_0(x')) , & t \in (0, \tau),  x' \in \partial \Omega, \\ \\
(u - U_0)(0,x) = 0, & x \in \Omega. 
\end{array}
\right.
$$
So, from Theorem \ref{th4.5} (III)-(IV), (\ref{eq5.4}) and Lemmata \ref{eq5.1}-\ref{eq5.2}, we obtain
$$
d(u,U_0) \leq C_2(R) \tau^{1/2}. 
$$
Choosing $\tau$ such that $C_2(R) \tau^{1/2} \leq R$, we obtain that $U \to u(U)$ maps $\mM(\tau,R)$ into itself. Next, we show that, if $\tau$ is sufficiently small, it is a contraction in $\mM(\tau,R)$. Let $U, V \in \mM(\tau,R)$. We indicate with $u$ and $v$ the corresponding solutions to (\ref{eq5.3}). Then we have
$$
\left\{\begin{array}{cc}
D_t (u - v)(t,x) = \sum_{|\alpha| = 2}Êa_\alpha(t,x, V(t,x), \nabla_x V(t,x)) D_x^\alpha (u - v)(t,x) \\ \\
+ \sum_{|\alpha| = 2} [a_\alpha(t,x, U(t,x), \nabla_x U(t,x)) - a_\alpha(t,x, V(t,x), \nabla_x V(t,x))] D_x^\alpha u(t,x) \\ \\
 + f(t,x, U(t,x)) - f(t,x,V(t,x)),  & t \in (0, \tau), x \in \Omega, \\ \\
D_t (u - v)(t,x') +  \sum_{j=1}^nÊb_j(t,x', V(t,x')) D_{x_j} (u - v) (t,x')  \\ \\
+  \sum_{j=1}^nÊ[b_j(t,x', U(t,x')) - b_j(t,x', V(t,x'))] D_{x_j} u(t,x') \\ \\
= h(t,x', U(t,x')) - h(t,x, V(t,x')) , & t \in (0, \tau),  x' \in \partial \Omega, \\ \\
(u - v)(0,x) = 0, & x \in \Omega. 
\end{array}
\right.
$$
Arguing as before, we obtain
$$
\begin{array}{c}
d(u,v) \\ \\
 \leq C_3(R) \tau^{1/2}(\sum_{|\alpha| = 2} \|a_\alpha(t,x, U(t,x), \nabla_x U(t,x)) - a_\alpha(t,x, V(t,x), \nabla_x V(t,x))] D_x^\alpha u \|_{C^{\beta/2,\beta}((0, \tau) \times \Omega)} \\ \\
+ \|f(t,x, U(t,x)) - f(t,x,V(t,x))\|_{C^{\beta/2,\beta}((0, \tau) \times \Omega)} + \\ \\
 + \sum_{j=1}^nÊ\|[b_j(t,x', U(t,x')) - b_j(t,x', V(t,x'))] D_{x_j} u(t,x') \|_{C^{\beta/2,1+\beta}((0, \tau) \times \partial \Omega)}Ê+ 
 \\ \\
+ \|h(t,x', U(t,x')) - h(t,x, V(t,x'))\|_{C^{\beta/2,1+\beta}((0, \tau) \times \partial \Omega)}).
\end{array}
$$
We have, for $|\alpha| = 2$
$$
\begin{array}{c}
\sum_{|\alpha| = 2} \|a_\alpha(t,x, U(t,x), \nabla_x U(t,x)) - a_\alpha(t,x, V(t,x), \nabla_x V(t,x))] D_x^\alpha u \|_{C^{\beta/2,\beta}((0, \tau) \times \Omega)}  \\ \\
\leq C_1 \sum_{|\alpha| = 2} \|a_\alpha(t,x, U(t,x), \nabla_x U(t,x)) - a_\alpha(t,x, V(t,x), \nabla_x V(t,x))\|_{C^{\beta/2,\beta}((0, \tau) \times \Omega)} \\ \\
\times  (\|u \|_{C^{\beta/2}((0, \tau);C^2)} + \|u \|_{C^{\beta/2,2+\beta}((0, \tau) \times \Omega)}) \\ \\
\leq C_4(R) d(U,V),
\end{array}
$$
$$
\|f(t,x, U(t,x)) - f(t,x,V(t,x))\|_{C^{\beta/2,\beta}((0, \tau) \times \Omega)}  \leq C_5(R) d(U,V), 
$$
$$
\begin{array}{c}
\sum_{j=1}^nÊ\|[b_j(t,x', U(t,x')) - b_j(t,x', V(t,x'))] D_{x_j} u(t,x') \|_{C^{\beta/2,1+\beta}((0, \tau) \times \partial \Omega)} \\ \\
\leq C_2 \sum_{j=1}^nÊ\|[b_j(t,x', U(t,x')) - b_j(t,x', V(t,x'))] \|_{C^{\beta/2,1+\beta}((0, \tau) \times \partial \Omega)} \|D_{x_j} u\|_{C^{\beta/2,1+\beta}((0, \tau) \times \partial \Omega)} \\ \\
\leq  C_3 \sum_{j=1}^nÊ\|[b_j(t,x', U(t,x')) - b_j(t,x', V(t,x'))] \|_{C^{\beta/2,1+\beta}((0, \tau) \times \partial \Omega)} \\ \\
\times ( \|u\|_{C^{\frac{1+\beta}{2}}((0, \tau); C^1(\Omega))} + \|u \|_{C^{\beta/2,2+\beta}((0, \tau) \times \Omega)}) \\ \\
\leq C_5(R) \|U-V\|_{C^{\beta/2,1+\beta}((0, \tau) \times \partial \Omega)} \leq C_5(R) d(U,V). 
\end{array}
$$
We conclude that
$$
d(u,v) \leq C_6(R) \tau^{1/2} d(U,V), 
$$
implying that $U \to u(U)$ is a contraction if $C_6(R) \tau^{1/2} < 1$. 

So the conclusion follows from the contraction mapping theorem.

\end{proof}

We deduce the following

\begin{theorem}\label{th3.5}
Consider system (\ref{eq5.1}), with the conditions (AG1)-(AG4). Let $u_0 \in C^{2+\beta}(\Omega)$ be such that 
\begin{equation}\label{eq3.2}
\begin{array}{c}
\sum_{|\alpha| = 2}Êa_\alpha(0,x', u_0(x'), \nabla_x u_0(x')) D_x^\alpha u_0(x') + f(0,x', u_0(x')) \\ \\
= -\sum_{j=1}^nÊb_j(0,x', u_0(x')) D_{x_j} u_0(x') + h(0,x', u_0(x')), \quad \forall x' \in \partial \Omega. 
\end{array}
\end{equation}
Then there exists $\tau \in (0, T]$ such that (\ref{eq5.1}) admits a unique solution $u$ in $C^{1+\beta/2,2+\beta}((0, \tau) \times \Omega)$, with $D_tu_{|(0, \tau) \times \partial \Omega} \in B((0, \tau); C^{1+\beta}(\partial \Omega))$.
\end{theorem}

\begin{proof}
The existence follows from Proposition \ref{pr5.4}. 

Concerning the uniqueness, let $u$ and $v$ be different solutions in $(0, \tau) \times \Omega$. Let $\tau_1:= \inf\{t \in [0, \tau]: u(t,\cdot) \neq v(t,\cdot)\}$. Then $0 \leq \tau_1 < \tau$ and $u(t,\cdot) = v(t,\cdot)$ $\forall t \in [0, \tau_1]$. We set $u_1:= u(\tau_1,\cdot) = v(\tau_1,\cdot)$. We consider the system
\begin{equation}\label{eq5.5}
\left\{\begin{array}{ll}
D_t z(t,x) = \sum_{|\alpha| = 2}Êa_\alpha(\tau_1+t,x, z(t,x), \nabla_x z(t,x)) D_x^\alpha z(t,x) + f(\tau_1+t,x, z(t,x), \nabla_x z(t,x)), & t \geq 0, x \in \Omega, \\ \\
D_t z(t,x') + \sum_{j=1}^nÊb_j(\tau_1+t,x', z(t,x')) D_{x_j} z(t,x') = h(\tau_1+t,x', z(t,x')), & t \geq 0,  x' \in \partial \Omega, \\ \\
z(0,x) = u_1(x), & x \in \Omega. 
\end{array}
\right.
\end{equation}
Then $u_1 = u(\tau_1,\cdot) \in C^{2+\beta}(\Omega)$. Moreover, if $x' \in \partial \Omega$, 
$$
\begin{array}{c}
\sum_{|\alpha| = 2}Êa_\alpha(\tau_1,x', u_1(x'), \nabla_x u_1(x')) D_x^\alpha u_1(x') + f(\tau_1,x', u_1(x'), \nabla_x u_1(x')) \\ \\
= \sum_{|\alpha| = 2}Êa_\alpha(\tau_1,x', u(\tau_1,x'), \nabla_x u(\tau_1,x')) D_x^\alpha u(\tau_1,x') + f(\tau_1,x', u(\tau_1,x'), \nabla_x u(\tau_1,x')) \\  \\
= - \sum_{j=1}^nÊb_j(\tau_1,x',u(\tau_1,x') ) D_{x_j} u(\tau_1,x') + h(\tau_1,x', u(\tau_1,x')) \\ \\
= - \sum_{j=1}^nÊb_j(\tau_1,x',u_1(x')) D_{x_j} u_1(x') + h(\tau_1,x', u_1(x')). 
\end{array}$$
We indicate with $U_1$ the unique solution in $C^{1+\beta/2,2+\beta}((0, T-\tau_1) \times \Omega)$ with $D_tU_{1|(0, T - \tau_1) \times \partial \Omega)}Ê\in B((0, T-\tau_1); C^{1+\beta}(\partial \Omega))$ of 
$$
\left\{\begin{array}{c}
D_t z(t,x) = \sum_{|\alpha| = 2}Êa_\alpha(\tau_1+t,x, u_1(x), \nabla_x u_1(x)) D_x^\alpha z(t,x) + f(\tau_1+t,x, u_1(x), \nabla_x u_1(x)), \\ \\
 t \in (0, T-\tau_1), x \in \Omega, \\ \\
D_t z(t,x') + \sum_{j=1}^nÊb_j(\tau_1+t,x', u_1(x')) D_{x_j} z(t,x') = h(\tau_1+t,x', u_1(x')), \\ \\
 t \in (0, T - \tau_1),  x' \in \partial \Omega, \\ \\
z(0,x) = u_1(x), \quad x \in \Omega. 
\end{array}
\right.
$$
existing by Theorem \ref{th4.5}. Arguing as in the proof of Proposition \ref{pr5.4}, we deduce that, $\forall R \in \R^+$ there exists $\tau(R) \in \R^+$ such that, if $0 < \delta \leq \tau(R)$, (\ref{eq5.5}) has a unique solution $z$ in 
$C^{1+\beta/2,2+\beta}((0, \delta) \times \Omega)$, with $D_tz_{|(0, \delta) \times \partial \Omega} \in B((0, \delta); C^{1+\beta}(\partial \Omega))$ and satisfying
$$
\|z - U_1\|_{C^{\beta/2}((0, \delta); C^1(\Omega))}Ê+ \|z - U_1\|_{B((0, \delta); C^{1+\beta}(\Omega))} \leq R. 
$$
Evidently, $u(\tau_1+ \cdot, \cdot)$ and $v(\tau_1+ \cdot, \cdot)$ both satisfy (\ref{eq5.5}) in $(0, \tau-\tau_1) \times \Omega$. Moreover, with the usual method we can see that
$$
\begin{array}{c}
\lim_{\delta \to 0}Ê\{\|u(\tau_1+\cdot,\cdot)  - U_1\|_{C^{\beta/2}((0, \delta); C^1(\Omega))}Ê+ \|u(\tau_1+\cdot,\cdot)  - U_1\|_{B((0, \delta); C^{1+\beta}(\Omega))} \\ \\
+ \|v(\tau_1+\cdot,\cdot)  - U_1\|_{C^{\beta/2}((0, \delta); C^1(\Omega))}Ê+ \|u(\tau_1+\cdot,\cdot)  - U_1\|_{B((0, \delta); C^{1+\beta}(\Omega))}\} = 0
\end{array}
$$
We deduce that there exists some $\delta$ in $(0, \tau - \tau_1]$ such that $u$ and $v$ coincide in $[\tau_1, \tau_1+\delta]Ê\times \Omega$, in contradiction with the definition of $\tau_1$. 

\end{proof}

\end{document}